\documentclass[11pt]{article}
\usepackage[T1]{fontenc}
\usepackage[a4paper,margin=1in]{geometry}
\usepackage{amsmath,amssymb,amsthm}
\usepackage[hyphens]{url}

\newtheorem{theorem}{Theorem}
\newtheorem{proposition}[theorem]{Proposition}
\newtheorem{lemma}[theorem]{Lemma}
\newtheorem{corollary}[theorem]{Corollary}
\theoremstyle{remark}
\newtheorem{remark}[theorem]{Remark}
\title{Generic degrees of real polynomial Keller maps with non-dense image}
\author{Piotr Migus}
\date{}

\begin{document}
\maketitle

\begingroup
\makeatletter
\renewcommand{\@makefntext}[1]{\noindent #1}
\renewcommand{\thefootnote}{}

\footnotetext{%
	\emph{E-mail:} \texttt{migus.piotr@gmail.com}%
}

\addtocounter{footnote}{-1}
\makeatother
\endgroup

\begin{abstract}
We determine the possible generic degrees of real polynomial Keller maps with non-dense image, where density is understood in the Euclidean topology. For every dimension $n\geq3$, these degrees are exactly the even integers
$d\geq4$. They are realized in dimension three by an explicit family
$G_d\colon\mathbb{R}^{3}\to\mathbb{R}^{3}$, and hence in every higher
dimension by stabilization. In dimension two, any such degree is an even integer at least six, and no such maps exist if the planar Jacobian conjecture holds; in dimension one, none exist. For the family $G_d$, we describe the image exactly and show that no $G_d$ omits a half-space. We also show that the maximal cardinality of a real fibre is not determined by the generic degree and is not uniformly bounded in this class.
\end{abstract}

\section{Introduction}

The Jacobian conjecture asserts that a polynomial map
$\mathbb C^{n}\to\mathbb C^{n}$ with constant nonzero Jacobian determinant,
a \emph{Keller map}, is invertible. In July 2026, an explicit
counterexample in dimension three was announced in \cite{Al} and verified
in \cite{CE}; see also \cite{Ta}. With $A=1+xy$ and
$B=A^{2}z+y^{2}(4+3xy)$, the map
\[
F=(P,Q,R),\qquad P=AB,\quad Q=y+3xB,\quad R=2x-3x^{2}y-x^{3}z
\]
has $\det\operatorname{Jac}F\equiv-2$, is defined over $\mathbb Q$, and is
not injective, already over $\mathbb R$, since
\[
F(0,0,-\tfrac14)=F(1,-\tfrac32,\tfrac{13}{2})=F(-1,\tfrac32,\tfrac{13}{2})=(-\tfrac14,0,0).
\]
The strong real Jacobian conjecture, in which the Jacobian is assumed only
to be nowhere vanishing, was disproved by Pinchuk's planar example
\cite{Pi}, whose Jacobian is non-constant. The constant-Jacobian version
is likewise false in dimension three, as shown by $F$.

The manuscript \cite{CE} embeds $F$ in a family $F_{\eta}$ of Keller maps with $\det\operatorname{Jac}F_{\eta}\equiv-2$, and \cite[Cor.~5.3]{CE} realizes \emph{every} generic degree $\geq3$ over $\mathbb C$ by such maps; Gallagher \cite{Ga} exhibited, from a separate construction, an explicit degree-four real Keller map whose real fibre is empty along a slice of the target, and likewise a family realizing every generic degree $\geq3$ over $\mathbb{C}$. Both precede this note. Fernandes and Jelonek \cite{FJ} had earlier produced non-dense real examples in two wider categories: a polynomial local diffeomorphism of $\mathbb R^{2}$ of nonconstant Jacobian \cite[Thm.~5]{FJ}, and a regular non-polynomial map of $\mathbb R^{3}$ of constant Jacobian \cite[Rem.~8]{FJ}. The question addressed here is which generic degrees occur for a real
Keller map with non-dense image. We show that in every dimension
$n\geq3$ they are exactly the even integers at least four. The lower
bound combines an elementary parity argument, used by Campbell
\cite{Ca13}, with his 1973 corollary excluding degree two; realization is
provided by explicit real members of the family of \cite{CE}.

\section{Generic degrees}

For each even $d\geq4$ let $G_{d}=(P,Q_{d},R_{d})$ be the member of the family of \cite{CE}, written out in \S4, with $\eta_{d-2}=1$ and all other $\eta_{j}=0$:
\[
Q_{d}=Q+A^{2}x^{d-2}B^{d},\qquad R_{d}=R-\tfrac{d-2}{d}\,x^{d}B^{d};
\]
these are polynomials, with component degrees
\[
(\deg P,\deg Q_d,\deg R_d)=(7,\,6d+2,\,6d).
\]
A real polynomial map is Keller if its complexification $H_{\mathbb C}$ is; a Keller map is dominant, its differential being everywhere invertible, and $\mu(H)=[\mathbb C(x_{1},\dots,x_{n}):\mathbb C(H_{1},\dots,H_{n})]$ is the generic degree, equal in characteristic zero to the number of points of a general fibre. Write $\mathrm{JC}_{2}$ for the two-dimensional complex Jacobian conjecture.

\begin{theorem}\label{main}
For $n\geq1$ let $S_{n}\subseteq\mathbb Z_{>0}$ be the set of generic degrees of real polynomial Keller maps $\mathbb R^{n}\to\mathbb R^{n}$ with non-dense image. Then
\[
S_{n}=\{2m:m\geq2\}\ \ (n\geq3),\qquad S_{1}=\varnothing,\qquad
S_{2}\subseteq\{2m:m\geq3\}\ \text{ with }\ \mathrm{JC}_{2}\Rightarrow S_{2}=\varnothing .
\]
Every degree $2m$ $(m\geq2)$ is realized in dimension three by $G_{2m}$, and in every $n\geq3$ by its stabilization $G_{2m}\times\mathrm{id}_{\mathbb R^{n-3}}$; composing with the linear change $L(p,q,r)=(r/2,q,p)$ of \cite[Cor.~3.2]{CE} makes the examples volume-preserving.
\end{theorem}

The argument for Proposition~\ref{prop} is a specialization of \cite[Thm.~5.2]{CE} to the present real family. We include it for completeness and to make the real fibre correspondence explicit.

\begin{proposition}\label{prop}
For even $d\geq4$, the map $G_{d}$ is Keller with $\det\operatorname{Jac}G_{d}\equiv-2$ and $\mu(G_{d})=d$. On $\{A\neq0\}$, with $s=x/A$, a point $(x,y,z)$ lies in $G_{d}^{-1}(p,q,r)$ only if
\[
\Omega_{d;\,p,q,r}(s):=\tfrac{2}{d}\,p^{d}s^{d}+2ps^{3}-qs^{2}+2s-r=0 ,
\]
and every simple root produces a unique preimage, via formulas with real coefficients; so over a target with $p\neq0$ (where no preimage has $A=0$, as $A=0$ forces $P=AB=0$) the real preimages are the real simple roots of $\Omega_{d;p,q,r}$.
\end{proposition}

\begin{proof}
On $\{A\neq0\}$ put $s=x/A$. Since $A$ is independent of $z$, $P_{z}=A^{3}$, and $s_{x}=A^{-2}$; expanding along the $y$-row,
\[
\det\frac{\partial(P,y,s)}{\partial(x,y,z)}=-P_{z}s_{x}=-A .
\]
From $A^{2}R=2xA-x^{2}y-x^{3}B$ one gets $R=2s-ys^{2}-Ps^{3}$ on $\{A\neq0\}$, and since $A^{2}x^{d-2}B^{d}=P^{d}s^{d-2}$ and $x^{d}B^{d}=P^{d}s^{d}$, the added terms give $Q_{d}=y+\varphi(P,s)$, $R_{d}=2s-ys^{2}-\theta(P,s)$ with
\[
\varphi(P,s)=3Ps+P^{d}s^{d-2},\qquad \theta(P,s)=Ps^{3}+\tfrac{d-2}{d}P^{d}s^{d}.
\]
A direct check gives $\theta_{s}=s^{2}\varphi_{s}$, whence in the coordinates $(P,y,s)$,
\[
\det\frac{\partial(P,Q_{d},R_{d})}{\partial(P,y,s)}=2-2ys-\theta_{s}+s^{2}\varphi_{s}=2(1-ys)=\frac{2}{A}
\]
(the last equality since $A(1-ys)=A-xy=1$). The product of the two determinants is $-2$ on the dense set $\{A\neq0\}$; as $\det\operatorname{Jac}G_{d}$ is a polynomial, $\det\operatorname{Jac}G_{d}\equiv-2$. In particular the Jacobian is a nonzero constant, so $G_{d}$ is dominant.

With $p=P$ and $y=q-\varphi(p,s)=q-3ps-p^{d}s^{d-2}$, substituting into $R_{d}=r$ gives $\Omega_{d;p,q,r}(s)=0$. Conversely, for a root with $D:=1-sy\neq0$, set $x=s/D$ and $z=pD^{3}-y^{2}(4-sy)D$; then $A=1+xy=1/D$ and $B=pD$, whence $P=AB=p$, $3xB=3ps$ and $A^{2}x^{d-2}B^{d}=p^{d}s^{d-2}$, so $Q_{d}=y+3ps+p^{d}s^{d-2}=q$, and $R_{d}=r$; thus $G_{d}(x,y,z)=(p,q,r)$ with $x/A=s$. As $\Omega_{d}'(s)=2D$, a root is simple exactly when $D\neq0$, and (all these formulas having real coefficients) $(x,y,z)\mapsto x/A$ is a bijection between $G_{d}^{-1}(p,q,r)\cap\{A\neq0\}$ and the simple roots of $\Omega_{d;p,q,r}$, real preimages corresponding to real roots. In particular $x,y,z$ are rational functions of $(p,q,r,s)$, while $p,q,r$ are polynomials in $x,y,z$ and $s=x/A$, so $\mathbb C(x,y,z)=\mathbb C(p,q,r)(s)$, and $\mu(G_{d})$ is the degree over $\mathbb C(p,q,r)$ of the minimal polynomial of $s$. That polynomial is $\Omega_{d}=f_{p,q}(s)-r$, $f_{p,q}(S)=\tfrac2d p^{d}S^{d}+2pS^{3}-qS^{2}+2S$: it is primitive in $s$ over $\mathbb C[p,q,r]$ (the coefficient of $s$ is $2$) and, being of degree one in $r$ with unit leading coefficient, irreducible in $\mathbb C[p,q,r,s]$, so by Gauss's lemma irreducible over $\mathbb C(p,q,r)$. Hence $\mu(G_{d})=\deg_{s}\Omega_{d}=d$. Finally $A=0$ forces $P=AB=0$, so over $p\neq0$ every real preimage has $A\neq0$ and gives a real root.
\end{proof}

\begin{proof}[Proof of Theorem~\ref{main}]
\emph{Necessity.} Let $H\colon\mathbb R^{n}\to\mathbb R^{n}$ be a real Keller map with non-dense image and $d=\mu(H)$. There is a nonempty Zariski-open $U\subseteq\mathbb C^{n}$, defined over $\mathbb R$, over which $H_{\mathbb C}$ has exactly $d$ distinct preimages (the complement of the branch and non-properness loci, which are conjugation-invariant since $H$ is real); so, $\mathbb C^{n}\setminus U$ being contained in the zero set of a non-zero real polynomial, its real points $U(\mathbb R)$ are Euclidean-dense in $\mathbb R^{n}$. As $H(\mathbb R^{n})$ is not dense, some nonempty Euclidean-open $V$ is disjoint from it; pick $w\in V\cap U(\mathbb R)$. Complex conjugation acts on the $d$-point fibre $H_{\mathbb C}^{-1}(w)$ with no fixed point, since a fixed point would be a real preimage and $H^{-1}(w)=\varnothing$; the points therefore fall into conjugate pairs and $d$ is even. (This is the parity principle used by Campbell \cite{Ca13}, in the form that odd-degree extension implies dense image \cite[\S3.4]{Ca12}.) Moreover $d\neq2$ by \cite[Cor., p.~248]{Ca}. Hence $d\geq4$ in every dimension, and $S_{n}\subseteq\{2m:m\geq2\}$ for all $n$.

\emph{Attainment.} Fix even $d\geq4$ and take $p=1$, $q=r=-M$ in Proposition~\ref{prop}:
\[
\Omega_{d;\,1,-M,-M}(s)=\tfrac{2}{d}s^{d}+2s^{3}+2s+M(s^{2}+1).
\]
As $d\geq4$ is even, $h_{d}(s):=\bigl(\tfrac2d s^{d}+2s^{3}+2s\bigr)/(s^{2}+1)\to+\infty$ as $|s|\to\infty$, so $h_{d}$ attains a finite minimum $m_{d}$; for $M>-m_{d}$ we get $\Omega_{d;1,-M,-M}=(s^{2}+1)(M+h_{d})>0$ on $\mathbb R$, so this fibre is empty. (For $d=4$, $M=\tfrac52$ gives $\Omega_{4;1,-5/2,-5/2}=\tfrac12\bigl((s^{2}+2s)^{2}+(s+2)^{2}+1\bigr)$.) Emptiness is open in the target: the continuous function $\Omega_{d,w_{0}}(s)/(1+s^{2})^{d/2}$, $w_{0}=(1,-M,-M)$, is positive with limit $\tfrac2d>0$ at infinity, hence $\geq c$ for some $c>0$; and, using $|s|^{k}\leq(1+s^{2})^{d/2}$ for $0\leq k\leq d$,
\[
\bigl|\Omega_{d,w}(s)-\Omega_{d,w_{0}}(s)\bigr|\leq\varepsilon(w)(1+s^{2})^{d/2},\qquad
\varepsilon(w)=\tfrac2d|p^{d}-1|+2|p-1|+|q+M|+|r+M|\xrightarrow[w\to w_{0}]{}0 .
\]
On a ball about $w_{0}$ we have $\varepsilon(w)<c$ and $p\neq0$, so $\Omega_{d,w}>0$ on $\mathbb R$ and $G_{d}^{-1}(w)=\varnothing$; thus $G_{d}(\mathbb R^{3})$ omits a neighbourhood of $w_{0}$, and $d\in S_{3}$. The stabilization $(x,y,z,u)\mapsto(G_{d}(x,y,z),u)$ is Keller of generic degree $d$ with image $G_{d}(\mathbb R^{3})\times\mathbb R^{n-3}$, so $d\in S_{n}$ for all $n\geq3$; hence $S_{n}=\{2m:m\geq2\}$ there. Finally $\det L=-\tfrac12$, so $L\circ G_{d}$ has Jacobian determinant $1$.

\emph{Small dimensions.} For $n=1$ a Keller map has constant nonzero derivative, hence is affine and a bijection: $S_{1}=\varnothing$. For $n=2$, if $H$ is a real Keller map with non-dense image and $\mathrm{JC}_{2}$ held, then $H_{\mathbb C}$ would be a polynomial automorphism of $\mathbb C^{2}$; being defined over $\mathbb R$, its inverse also has real coefficients, so $H$ would be a bijection of $\mathbb R^{2}$, contradicting non-density. Hence $\mathrm{JC}_{2}\Rightarrow S_{2}=\varnothing$, while $S_{2}\subseteq\{2m:m\geq2\}$ is the dimension-free necessity above. Unconditionally, the generic degree of a Keller map of
$\mathbb{C}^{2}$ cannot equal two or three \cite{Or86}, nor four
\cite{Do00}, with the case of one dicritical component treated in
\cite{DO98}. Together with the parity argument, this gives
\[
S_{2}\subseteq\{2m:m\geq3\}.
\]
\end{proof}

\section{The image of $G_d$}

Throughout, $d\geq4$ is even, $w=(p,q,r)$, and $f_{p,q}(s)=\tfrac2d\,p^{d}s^{d}+2ps^{3}-qs^{2}+2s$
is the polynomial from the proof of Proposition~\ref{prop}, so that $\Omega_{d;w}=f_{p,q}-r$; let
$N_{d}(w)=\#\,G_{d}^{-1}(w)$.

\begin{lemma}\label{lem:desc}
	For $p\neq0$ the polynomial $\Omega_{d;w}'(s)=2p^{d}s^{d-1}+6ps^{2}-2qs+2$ has at most three real
	zeros counted with multiplicity.
\end{lemma}

\begin{proof}
	As $d$ is even and $p\neq0$ we have $2p^{d}>0$, and the constant term $2$ is non-zero, so $0$ is
	not a zero. Since $d\geq4$ the exponents $d-1\geq3$, $2$, $1$, $0$ are pairwise distinct, so the
	coefficient sequence is $(2p^{d},6p,-2q,2)$: at most four monomials, independently of $d$.
	Descartes' rule of signs, which bounds the number of positive zeros counted with multiplicity by
	the number of sign changes, applied to $\Omega'(s)$ and to
	$\Omega'(-s)=-2p^{d}s^{d-1}+6ps^{2}+2qs+2$ ($d-1$ being odd), gives a total of at most three in
	each of the six cases according to the signs of $p$ and $q$.
\end{proof}

\begin{lemma}\label{lem:M}
	For $p\neq0$ the number $M$ of real zeros of $\Omega_{d;w}$, counted with multiplicity, lies in
	$\{0,2,4\}$; in particular $N_{d}(w)\leq4$.
\end{lemma}

\begin{proof}
	Let $t_{1}<\dots<t_{\ell}$ be the distinct real zeros with multiplicities $m_{1},\dots,m_{\ell}$
	and $M=\sum m_{i}$. Then $\Omega'$ vanishes at $t_{i}$ to order $m_{i}-1$ and, by Rolle, at least
	once in each of the $\ell-1$ intervals $(t_{i},t_{i+1})$, these points being distinct from the
	$t_{i}$; so $\Omega'$ has at least $\sum(m_{i}-1)+(\ell-1)=M-1$ real zeros with multiplicity,
	whence $M-1\leq3$ by Lemma~\ref{lem:desc}. Finally $M\equiv\deg_{s}\Omega_{d;w}=d\equiv0\pmod2$.
	By Proposition~\ref{prop}, $N_{d}(w)$ is the number of \emph{simple} real zeros, at most
	$M\leq4$.
\end{proof}

\begin{proposition}\label{prop:pzero}
	Over $p=0$ the preimages with $A=0$ form a single point when $q\neq0$, namely $x=2/q$, $y=-q/2$
	with $z$ determined by $r$, and none when $q=0$; the preimages with $A\neq0$ are the simple real
	zeros of $-qs^{2}+2s-r$. Hence $N_{d}(0,q,r)\leq3$, and the real fibre over $\{p=0\}$ is never
	empty.
\end{proposition}

\begin{proof}
	If $A=0$ then $xy=-1$, so $B=y^{2}$, $P=AB=0$ and, the correction term $A^{2}x^{d-2}B^{d}$
	vanishing, $Q_{d}=y+3xB=y+3xy^{2}=-2y$; thus $y=-q/2$ and $x=2/q$, which requires $q\neq0$. Then
	$xB=q/2$ and $2x-3x^{2}y=10/q$, so $R_{d}=10/q-8z/q^{3}-\tfrac{d-2}{d}(q/2)^{d}$, affine and
	non-constant in $z$. If $A\neq0$, the correspondence of Proposition~\ref{prop} applies without
	the hypothesis $p\neq0$, and at $p=0$ the fibre polynomial degenerates to $-qs^{2}+2s-r$.
	Counting: at most one preimage with $A=0$ and at most two with $A\neq0$; and if $q=0$ the linear
	polynomial $2s-r$ has a simple zero, while if $q\neq0$ the $A=0$ preimage exists.
\end{proof}

Since $\det\operatorname{Jac}G_{d}\equiv-2$, the map $G_{d}$ is a local diffeomorphism and
$G_{d}(\mathbb{R}^{3})$ is open; this is why non-density, rather than non-surjectivity, is the
right condition. For $p\neq0$ the polynomial $f_{p,q}$ has even degree $d$ with positive leading
coefficient $\tfrac2d p^{d}$, so it attains a global minimum; put
$\beta_{d}(p,q):=\min_{s\in\mathbb{R}}f_{p,q}(s)$.

\begin{lemma}\label{lem:beta}
	$\beta_{d}$ is continuous on $\{p\neq0\}$.
\end{lemma}

\begin{proof}
	Let $K\subset\{p\neq0\}$ be compact, so $|p|\geq c>0$ there and the remaining coefficients are
	bounded. As $d\geq4$, every exponent occurring in $f_{p,q}$ other than $d$ is at most $d-1$, so
	$f_{p,q}(s)\geq\tfrac2d c^{d}|s|^{d}-C(K)|s|^{d-1}$ for $|s|\geq1$, the right side tending to
	$+\infty$ uniformly on $K$. Since $f_{p,q}(0)=0$ there is $R=R(K)$ with $f_{p,q}(s)>0$ for
	$|s|>R$ and all $(p,q)\in K$, so $\beta_{d}=\min_{|s|\leq R}f_{p,q}$ is the minimum of a jointly
	continuous function on a compact set.
\end{proof}

\begin{proposition}\label{prop:image}
	For $p\neq0$ the real fibre of $G_{d}$ over $(p,q,r)$ is non-empty if and only if
	$r>\beta_{d}(p,q)$. Consequently
	\[
	G_{d}(\mathbb{R}^{3})=\{p=0\}\cup\{(p,q,r):p\neq0,\ r>\beta_{d}(p,q)\},
	\]
	and $\Theta_{d}:=\mathbb{R}^{3}\setminus G_{d}(\mathbb{R}^{3})
	=\{(p,q,r):p\neq0,\ r\leq\beta_{d}(p,q)\}$ is unbounded, has non-empty interior, and contains a
	downward vertical ray over every $(p,q)$ with $p\neq0$; it is closed in $\mathbb{R}^{3}$, the
	image being open.
\end{proposition}

\begin{proof}
	If $r<\beta_{d}(p,q)$ then $\Omega=f_{p,q}-r>0$ on $\mathbb{R}$ and $N_{d}=0$. If
	$r=\beta_{d}(p,q)$ then every real zero of $\Omega$ is a global minimiser of $f_{p,q}$ and
	$\Omega\geq0$ there, so its multiplicity is even; there are no simple real zeros and $N_{d}=0$.
	If $r>\beta_{d}(p,q)$, let $s_{m}$ realize the minimum; then $\Omega$ is positive at $\pm\infty$
	and negative at $s_{m}$, hence changes sign at least twice, hence has at least two distinct zeros
	of odd multiplicity; their multiplicities sum to at most $M\leq4$ by Lemma~\ref{lem:M}, and two
	odd multiplicities $\geq3$ would sum to at least six, so at least one zero is simple and
	$N_{d}\geq1$. The description of the image follows since the real fibre over $\{p=0\}$ is never
	empty (Proposition~\ref{prop:pzero}); the vertical rays and unboundedness of $\Theta_{d}$ are
	immediate from the displayed description, its interior is non-empty by Lemma~\ref{lem:beta}, and
	it is closed because the image is open.
\end{proof}

The description may look inconsistent with openness, $\Theta_{d}$ being contained in $\{p\neq0\}$
and yet closed, and containing a downward ray over every $p\neq0$. It is not: since $\{p=0\}$
lies in the image and the image is open, $\beta_{d}$ diverges as $p\to0$: given $C>0$ and
$q_{0}$, a ball about $(0,q_{0},-C)$ lies in the image, so $\beta_{d}(p,q)<-C$ for $p\neq0$ near
$0$ and $q$ near $q_{0}$; compactness of $[-Q,Q]$ makes this uniform for $|q|\leq Q$.

Both sets are semialgebraic by Tarski--Seidenberg and $\beta_{d}$ is a semialgebraic function, so
Proposition~\ref{prop:image} is an explicit quantifier elimination for this family rather than a
new kind of statement. For $d=4$ and $(p,q)=(1,-\tfrac52)$ one has
$f'_{1,-5/2}(s)=2s^{3}+6s^{2}+5s+2=(s+2)(2s^{2}+2s+1)$, the quadratic factor having negative
discriminant, so $s=-2$ is the only critical point and $\beta_{4}(1,-\tfrac52)=-2$ exactly: the
whole ray $r\leq-2$ over $(1,-\tfrac52)$ is omitted.

\begin{corollary}\label{cor:half}
	No $G_{d}$ omits a half-space.
\end{corollary}

\begin{proof}
	Suppose $H=\{w:\langle v,w\rangle\geq c\}\subseteq\Theta_{d}$ with $v\neq0$. Since
	$\Theta_{d}\cap\{p=0\}=\varnothing$, the functional $\langle v,\cdot\rangle$ is $<c$ on the plane
	$\{p=0\}$; a linear functional bounded above on a $2$-plane through the origin vanishes on it, so
	$v=(v_{p},0,0)$ with $v_{p}\neq0$ and $H$ is $\{p\geq t\}$ or $\{p\leq t\}$. Such an $H$ contains
	points with $p\neq0$ fixed and $r$ arbitrarily large, which lie in the image by
	Proposition~\ref{prop:image}.
\end{proof}

The boundary found in Proposition~\ref{prop:image} is the graph of $\beta_{d}$ over $\{p\neq0\}$
(by the divergence just noted, it has no limit points over $\{p=0\}$): a surface. That its
dimension is $n-1$ is not an accident of this family:

\begin{remark}\label{rem:bdim}
	For any polynomial local diffeomorphism $H\colon\mathbb{R}^{n}\to\mathbb{R}^{n}$ with non-dense image, $\dim\partial H(\mathbb{R}^{n})=\dim S_{H}=n-1$, where $S_{H}$ is the set of points at which $H$ is not proper. Indeed $\partial H(\mathbb{R}^{n})$ is closed, semialgebraic and contained in $S_{H}$ \cite[\S1]{Je02}, and $\dim S_{H}\leq n-1$ \cite[Thm.~6.4]{Je02}; were its codimension at least two, $\mathbb{R}^{n}\setminus\partial H(\mathbb{R}^{n})$ would be connected \cite[Lem.~8.1]{Je02} and hence filled by the open and closed set $H(\mathbb{R}^{n})$, contradicting non-density.
\end{remark}

\section{Real fibre cardinality}

Let $\eta=(\eta_j)_{j\geq1}$ be a finitely supported sequence of real
numbers, and write $G_{\eta}=(P,Q_{\eta},R_{\eta})$ for the corresponding
member of the family of \cite{CE}, where
\[
Q_{\eta}
=Q+\sum_{j\geq1}\eta_j A^{2}x^{j}B^{j+2},
\qquad
R_{\eta}
=R-\sum_{j\geq1}\eta_j\frac{j}{j+2}x^{j+2}B^{j+2}.
\]
Thus $G_d$ is obtained by taking $\eta_{d-2}=1$ and all other
$\eta_j=0$. Its fibre polynomial is
\[
\Omega_{\eta;p,q,r}(s)
=\sum_{j\geq1}\eta_j\frac{2}{j+2}p^{j+2}s^{j+2}
+2ps^{3}-qs^{2}+2s-r.
\]
For $w\in\mathbb R^{3}$, write
\[
N_{\eta}(w)=\#\,G_{\eta}^{-1}(w).
\]

\begin{proposition}\label{prop:allreal}
	Let $d\geq4$ be even. Choose distinct reals $a_{1},\dots,a_{d}$ such that the coefficient of $s$
	in $g(s)=\prod_{i}(s-a_{i})$ is non-zero; replacing every $a_{i}$ by $-a_{i}$, which changes the
	sign of that coefficient because $d$ is even, we may assume it positive. Such a choice exists for
	every even $d\geq4$: for $a_{i}=-i$ one has $[s^{1}]g=d!\,H_{d}>0$, with $H_{d}$ the $d$-th
	harmonic number. Put $\lambda=2/[s^{1}]g>0$, $\Omega=\lambda g$, and set $p=1$,
	$q=-[s^{2}]\Omega$, $r=-[s^{0}]\Omega$, $\eta_{j}=\tfrac{j+2}{2}[s^{j+2}]\Omega$ for
	$2\leq j\leq d-2$, $\eta_{1}=\tfrac32\bigl([s^{3}]\Omega-2\bigr)$, and all other $\eta_{j}=0$.
	Then $G_{\eta}$ is a Keller map with $\det\operatorname{Jac}G_{\eta}\equiv-2$,
	$\eta_{d-2}=\lambda d/2>0$ and $\mu(G_{\eta})=d$, its image is non-dense, and the real fibre over
	$(1,q,r)$ has full cardinality, $N_{\eta}(1,q,r)=\mu(G_{\eta})=d$.
\end{proposition}

\begin{proof}
	The argument of Proposition~\ref{prop} extends to $G_{\eta}$. Indeed, on
	$\{A\neq0\}$ one has
	$Q_{\eta}=y+\varphi(P,s)$ and $R_{\eta}=2s-ys^{2}-\theta(P,s)$ with
	$\varphi(P,s)=3Ps+\sum_{j}\eta_{j}P^{j+2}s^{j}$ and
	$\theta(P,s)=Ps^{3}+\sum_{j}\eta_{j}\tfrac{j}{j+2}P^{j+2}s^{j+2}$, the added terms
	$A^{2}x^{j}B^{j+2}$ and $x^{j+2}B^{j+2}$ being polynomials; these satisfy
	$\theta_{s}=s^{2}\varphi_{s}$ term by term, so the determinant computation gives
	$\det\operatorname{Jac}G_{\eta}\equiv-2$, and the correspondence uses only
	$\partial_{s}\Omega_{\eta}=2(1-sy)=2D$, which holds for every $\eta$. That
	$\mu(G_{\eta})=\deg_{s}\Omega_{\eta}=d$ follows exactly as there: $\Omega_{\eta}$ has degree one
	in $r$ with unit leading coefficient, hence is irreducible in $\mathbb{C}[p,q,r,s]$, and it is
	primitive in $s$ because the coefficient of $s$ is $2$, so it is irreducible over
	$\mathbb{C}(p,q,r)$ by Gauss's lemma. This recovers, in the present notation, the generic-degree computation of
	\cite[Thm.~5.2]{CE}. By construction $\Omega_{\eta;1,q,r}=\lambda
	g$ has $d$ distinct simple real zeros, and each yields a preimage since
	$\partial_{s}\Omega=2D\neq0$ at a simple zero; as $p=1\neq0$ no preimage has $A=0$, so these
	exhaust the fibre. For non-density, the leading coefficient of $\Omega_{\eta;p,q,r}$ in $s$ is
	$\eta_{d-2}\tfrac2d\,p^{d}>0$ for every $p\neq0$ because $d$ is even, so
	$f_{\eta;p,q}:=\Omega_{\eta;p,q,r}+r$ (independent of $r$) is coercive; its minimum $\beta_{\eta}(p,q)$ is continuous
	on $\{p\neq0\}$ by the argument of Lemma~\ref{lem:beta}, and
	$\{(p,q,r):p\neq0,\ r<\beta_{\eta}(p,q)\}$ is a non-empty open set on which $\Omega_{\eta}>0$,
	hence disjoint from the image.
\end{proof}

\begin{remark}\label{rem:dich}
	The maximal cardinality of a real fibre is not determined by the generic
	degree. Indeed, for every even $d\geq6$, the maps $G_d$ have generic degree
	$d$ and all their real fibres have cardinality at most four, whereas the
	preceding construction gives a map of the same generic degree with a real
	fibre of cardinality $d$. Consequently, there is no uniform bound on real
	fibre cardinalities for Keller maps $\mathbb{R}^{3}\to\mathbb{R}^{3}$ with
	non-dense image.
\end{remark}

The bound four for $G_{d}$ is a sparsity phenomenon read off the definition: $\Omega_{d;w}'$ has
at most four monomials for every $d$. It is not a statement about real Keller maps, and it uses that $d$
is even.

\begin{remark}\label{rem:final}
	The realization is minimal in generic degree, not in component degree:
	the component degrees of $G_4$ are $(7,26,24)$, compared with
	$(12,11,4)$ for Gallagher's map. The omitted set is described exactly
	by Proposition~\ref{prop:image}. By Corollary~\ref{cor:half} no $G_{d}$ omits a half-space, whereas the image of each example of \cite{FJ} lies in one: $\mathbb{R}\times\mathbb{R}_{+}$ for the polynomial map of nonconstant Jacobian \cite[Thm.~5]{FJ}, and $\mathbb{R}^{2}\times\mathbb{R}_{+}$ for the regular map of Jacobian one \cite[Rem.~8]{FJ}. Nor do the obstructions of \cite{Je02} exclude a half-space, a hyperplane being both unbounded and $\mathbb{R}$-uniruled. Whether a polynomial Keller map can omit a half-space is left open. By contrast, the image can never be the exterior of a ball; this follows
	from \cite[Ex.~7.3(b)]{Je02}, which applies to arbitrary polynomial maps.
\end{remark}

\section*{Acknowledgments} 
The author thanks Z.~Jelonek for a helpful comment. The author is partially supported by the grant of Narodowe Centrum Nauki, number 2024/55/B/ST1/01412.

\end{document}